\documentclass[11pt]{amsart}

\usepackage{amsmath}
\usepackage{amssymb}
\usepackage{amsfonts}
\usepackage{amsthm}
\usepackage{bbm}
\usepackage[T1]{fontenc}
\usepackage[usenames,dvipsnames]{xcolor}
\usepackage[colorlinks=true,linkcolor=NavyBlue,urlcolor=RoyalBlue,citecolor=PineGreen,bookmarks=true,bookmarksopen=true,bookmarksopenlevel=2,unicode=true,linktocpage]{hyperref}

\usepackage[a4paper,  left=25mm, right=25mm, top=33mm, bottom=33mm, marginparwidth= 20mm]{geometry}

\newtheorem{theorem}{Theorem}[section]

\newtheorem{proposition}[theorem]{Proposition}
\newtheorem{lemma}[theorem]{Lemma}

\newcommand{\1}{\mathbbm{1}}
\renewcommand{\C}{\mathbb{C}}
\newcommand{\cC}{\mathcal{C}}
\newcommand{\conf}{\mathsf{Conf}}
\newcommand{\CSX}{\cC(S\hspace{0.7pt};\hspace{-0.7pt}X)}
\renewcommand{\d}{{\; {\rm d}}}

\newcommand{\E}{\mathbb{E}}
\renewcommand{\epsilon}{\varepsilon}
\newcommand{\ext}{\mathchoice
             {\raisebox{1pt}{\text{$\textstyle\bigwedge$}}}             
             {\raisebox{1pt}{$\bigwedge$}}
             {\raisebox{0.5pt}{$\scriptstyle\bigwedge$}}
             {\raisebox{0.2pt}{$\scriptscriptstyle\bigwedge$}}}
\newcommand{\K}{\mathbb{K}}
\newcommand{\Part}{\mathcal P}
\newcommand{\proj}[1]{{\sf\Pi}^{#1}}
\newcommand{\R}{\mathbb{R}}
\newcommand{\rP}{\mathsf{P}}
\newcommand\sbinom[2]{\text{$\textstyle\binom{#1}{#2}$}}
\newcommand\set[1]{[\![#1]\!]}
\newcommand{\Vect}{{\mathrm{Vect}}}
\newcommand{\X}{{\mathsf X}}

\title[Mean projection for determinantal point processes]{On the mean projection theorem for\\determinantal point processes}

\author{Adrien Kassel}
\address{Adrien Kassel -- CNRS -- UMPA, ENS de Lyon}
\email{adrien.kassel@ens-lyon.fr}

\author{Thierry L\'evy}
\address{Thierry L\'evy -- LPSM, Sorbonne Universit\'e, Paris}
\email{thierry.levy@sorbonne-universite.fr}

\date{\today}



\begin{document}

\begin{abstract}In this short note, we extend to the continuous case a mean projection theorem for discrete determinantal point processes associated with a finite range projection, thus strengthening a known result in random linear algebra due to Ermakov and Zolotukhin. We also give a new formula for the variance of the exterior power of the random projection.
\end{abstract}

\maketitle

%%%%%%%%%%%%%%%%%%%%%%%%%%%%%%%%%%%%%%%%%%%%%

\section{Introduction}

Kirchhoff's work on electrical networks \cite{Kirchhoff} seems to be one of the earliest works in the literature where linear algebra and graph-theoretical combinatorial methods were put together. Later on, linear algebra problems, and classical determinantal methods for solving them, gave rise to various statistical approaches, notably linked to the so-called determinantal point processes (introduced by Macchi in 1975 \cite{Macchi}, and named like this by Borodin, only around 2000 which saw a blossoming of results on those processes from various authors, see \cite{Soshnikov, Shirai-Takahashi,Lyons-DPP, Johansson, Borodin}). These methods recently became an active field in randomized numerical linear algebra~\cite{Derezinski}.

 In his work, Kirchhoff solved a linear algebra system on an electrical network seen as a finite graph, by expressing the current induced by an external battery hooked on the network, as an average over spanning trees of a certain current associated to the tree. In modern terms, he expressed an orthogonal projection as the expectation of a certain random projection associated to a random spanning tree. 
Such a mean projection theorem appeared in several guises in the literature, and more or less independently, in works of Maurer \cite{Maurer}, Lyons \cite{Lyons-DPP}, Catanzaro--Chernyak--Klein \cite{CCK-line}, and probably others that we are unaware~of. 

In our work \cite[Theorem 5.9]{KL3}, we extended the mean projection formula for determinantal point processes on finite sets, thus putting the statements of \cite{Kirchhoff,Maurer,Lyons-DPP,CCK-line} in a unified geometric framework, and strengthening the result by proving a mean projection theorem for the exterior powers of the projections, that is, for minors of their matrices in a fixed basis. 

Let us quickly recall our statement. 
Let $\K$ be $\R$ or $\C$, let $E$ be a finite dimensional Euclidean space on $\K$ of dimension $d$, and let $(e_i)_{1\le i\le d}$ be an orthonormal basis of $E$. We let $S=\{1,\ldots, d\}$ and consider $H$ a subspace of $E$ of dimension $n$. Let $\X$ be the determinantal point process on $S$ associated to the matrix $K=\left(\langle e_i, \proj{H} e_j\rangle\right)_{1\le i,j\le d}$, where $\proj{H}$ is the orthogonal projection on $H$. For each $X\subseteq S$, let $E_X=\bigoplus_{x\in X} \K e_x$ be the corresponding coordinate subspace of $E$.

\begin{theorem}\label{thm:discrete}
Almost surely, the equality $E=H\oplus E_{\X}^\perp$ holds, and denoting by $\rP_{\X}$ the projection on $H$ parallel to $E_{\X}^\perp$, we have
\[\E\big[\ext \rP_{\X}\big]=\ext\proj{H}\,.\]
\end{theorem}

In words, in a fixed basis of $E$, the expectation of any minor of the matrix of $\rP_{\X}$ is equal to the same minor of $\proj{H}$.

A short while ago, it came to our attention while reading the recent statistics paper \cite{Bardenet} on Monte--Carlo integration methods, that such a mean projection formula had also appeared in~\cite{EZ} in the case of $S=\R$, in a different guise, although the relation to the above-cited works was not mentioned there. 

One of the referees of this paper kindly pointed out to us that results in the spirit of Theorem~\ref{thm:discrete} have also been obtained in the context of the resolution of singular linear systems of equations, for instance in \cite{Berg,BenTal-Teboulle} and more recently in the context of active sampling for linear regression \cite[Thms 5, 6 and 7]{Derezinski-Warmuth}, see also \cite{Avron-Boutsidis,Zelda-Suvrit,Derezinski-Warmuth-Hsu}. In \cite[Def. 4]{Derezinski-Liang-Mahoney}, the authors define the class of random matrices for which the expectation of any minor equals the same minor of the expectation, give basic properties, and provide a few examples. Theorem~\ref{thm:discrete} and \cite[Thm 5.9]{KL3} give families of examples of such random matrices, namely the matrices~$\rP_{\X}$. A systematic study of this class of random matrices would certainly be interesting. \nolinebreak
\smallskip

The goal of this short note is to extend Theorem~\ref{thm:discrete} to the case of a determinantal point process associated to a finite rank orthogonal projection on any Polish space~$S$, so that it applies for instance to any orthogonal polynomial ensemble, see \cite[Section 3.8]{Lyons-ICM}. This extension is the content of Theorem~\ref{thm:projection-continuous}. An extension of Theorem~\ref{thm:discrete} to the case of a projection with infinite range (both in the case where~$S$ is discrete or continuous) would be interesting. An example of this situation is investigated in \cite{Bufetov-Qiu}, where the author study among other things the continuous analogue of $\rP_{\X}$ in the case of the Bergman kernel.

%%%%%%%%%%%%%%%%%%%%%%%%%%%%%%%%%%%%%%%%%%%%%

\section{The mean projection theorem}

Let $S$ be a Polish space and $\lambda$ a positive Radon measure on $S$. Let us consider the space $E=L^{2}(S,\lambda)$ and the space $\cC(S)$ of continuous functions on $S$.\footnote{The space of continuous functions plays for us the role usually devoted to a reproducing kernel Hilbert space (RKHS), namely that of a space of functions that can be evaluated at points. However, we do not need this extra structure, because we do not need evaluation at a point to be a continuous linear form. Moreover, it seems that in many examples of interest, the RKHS is a subspace of continuous functions, so that our result applies.} Let $H\subseteq E\cap \cC(S)$ be a linear subspace of finite dimension $n$.

Let $\conf_n(S)$ be the set of collections of $n$ distinct points in $S$, and let $\mu$ be the determinantal probability measure on $\conf_n(S)$ associated with the orthogonal projection on $H$. This means that if we choose an orthonormal basis $(\varphi_j)_{1\le j\le n}$ of $H$, then we have for any bounded continuous symmetric test function $T:S^{n}\to \C$ the equality
\begin{equation}\label{eq:densitydpp}
\int_{\conf_n(S)} T(X)\, \d \mu(X)=\frac{1}{n!}\int_{S^{n}} T(x_{1},\ldots,x_{n}) \vert \det \big(\varphi_j(x_i)_{1\leq i,j \leq n}\big)\vert^{2} \, \d\lambda^{\otimes n}(x_{1},\ldots,x_{n}),
\end{equation}
in which the right-hand side does not depend on the choice of the orthonormal basis.
We will denote by $\X$ a random subset of $S$ distributed according to $\mu$, and use the notation $\E[T(\X)]$ for either of the two sides of the equality above. 

It follows from \eqref{eq:densitydpp} that $\mu$-almost every $X$ is a {\em uniqueness set} for $H$, in the sense that two elements of $H$ that coincide on $X$ are equal.\footnote{The uniqueness property is true for all determinantal processes associated with an orthogonal projection of possibly infinite range, that is with infinitely many points ($n=\infty$), as proved in the discrete case by Lyons~\cite{Lyons-DPP}, and recently by Bufetov-Qiu-Shamov~\cite{Bufetov-Qiu-Shamov} in the general case, following partial results by Ghosh~\cite{Ghosh}.
} This fact can be used to define a random projection onto $H$, as follows. For every $X\in \conf_n(S)$, let us define $\CSX=\{f\in \cC(S) : f_{|X}=0\}$.

\begin{lemma}\label{lem:decomposition}
For $\mu$-almost every $X\in \conf_n(S)$, the decomposition $\cC(S)=H\oplus \CSX$ holds.
\end{lemma}

\begin{proof} Let $f$ be an element of $\cC(S)$. Let $(\varphi_j)_{1\le j\le n}$ be an orthonormal basis of $H$. For $\mu$-almost every $X=\{x_{1},\ldots,x_{n}\}$ in $\conf_n(S)$, we have $\det (\varphi_j(x_i)_{1\leq i,j \leq n})\ne 0$, so that the system 
\[\alpha_{1}\varphi_{1}(x_{i})+\ldots + \alpha_{n}\varphi_{n}(x_{i})= f(x_i)\;, \quad \forall i\in \{1,\ldots, n\}\]
admits a unique solution. Then $\rP_X f=\alpha_{1}\varphi_{1}+\ldots + \alpha_{n}\varphi_{n}$
is the unique element of $H$ which takes the same values as $f$ on $X$. 
\end{proof}

For the rest of this note, we will keep the notation $\rP_{X}$ introduced in the previous proof for the projection on $H$ parallel to $\CSX$. Let us emphasize that the decomposition given by Lemma~\ref{lem:decomposition} depends on $H$ and $X$, but is independent of the Euclidean structure of $E$. In particular, the projection $\rP_{X}$ is independent of this Euclidean structure. 

For example, if $S=\R$, $\lambda$ is a measure with infinite support which admits moments of all orders, and $\varphi_{1},\ldots,\varphi_{n}$ are the first $n$ orthogonal polynomials with respect to $\lambda$, then $H$ is the space of polynomial functions of degree at most $n-1$ and $\rP_{X}f$ is the \emph{interpolating polynomial} of the restriction of $f$ to $X$.

For all $g_{1},\ldots,g_{m}\in E \cap \cC(S)$, let us define $g_{1}\wedge \ldots \wedge g_{m} \in L^{2}(S^{m},\tfrac{1}{m!}\lambda^{\otimes m})\cap \cC(S^{m})$ by setting, for all $y_{1},\ldots,y_{m}\in S$,
\begin{equation}\label{eq:defwedge}
(g_{1}\wedge \ldots \wedge g_{m})(y_{1},\ldots,y_{m})=\det\big(g_{j}(y_{i})_{1\leq i,j \leq m}\big).
\end{equation}
We will use several times the Andreieff--Heine identity, which is a continuous analogue of the Cauchy--Binet identity, and can be phrased as follows: if $h_{1},\ldots,h_{m}$ belong to $E\cap \cC(S)$, then 
\begin{equation}\label{eq:AH}
\langle g_{1}\wedge \ldots \wedge g_{m},h_{1}\wedge \ldots \wedge h_{m}\rangle_{L^{2}(S^{m},\tfrac{1}{m!}\lambda^{\otimes m})}=\det\big(\langle g_{i},h_{j}\rangle\big)_{1\leq i,j\leq m}.
\end{equation}
This equality justifies, for instance, the fact that the measure $\mu$ defined by \eqref{eq:densitydpp} is a probability measure.

Let us write $H^{0}=H$ and $H^{1}=H^{\perp}$. The isomorphism of vector spaces $L^{2}(S^{m},\frac{1}{m!}\lambda^{\otimes m})\simeq L^{2}(S,\lambda)^{\otimes m}$ is $\sqrt{m!}$ times an isometry, and the orthogonal decomposition $L^{2}(S)=H^{0}\oplus H^{1}$ gives rise to an orthogonal decomposition
\begin{equation}\label{eq:decL2}
L^{2}(S^{m})\simeq L^{2}(S)^{\otimes m}=\bigoplus_{\epsilon_{1},\ldots,\epsilon_{m}\in \{0,1\}}\!\! H^{\epsilon_{1}}\otimes \ldots \otimes H^{\epsilon_{m}}=\bigoplus_{k=0}^{m} \bigg[\bigoplus_{\substack{\epsilon_{1},\ldots,\epsilon_{m}\in \{0,1\}\\ \epsilon_{1}+\ldots+\epsilon_{m}=k}} \!\! H^{\epsilon_{1}}\otimes \ldots \otimes H^{\epsilon_{m}}\bigg].
\end{equation}
Let us denote by ${\sf\Pi}_{k}$ the orthogonal projection of $L^{2}(S^{m})$ on the $k$-th summand of the last expression. In order to describe this operator more concretely, recall that we denote by $\proj{H}$ the orthogonal projection on $H$ in $E$. For all real $t$, let us define the linear operator ${\sf D}_{t}=\proj{H}+t\proj{H^{\perp}}$ on $E$. Then 
\[{\sf D}_{t} g_{1}\wedge \ldots \wedge {\sf D}_{t} g_{m}=\sum_{k=0}^{m} t^{k}\, {\sf \Pi}_{k}(g_{1}\wedge \ldots \wedge g_{m}).\]
In words, ${\sf \Pi}_{k}(g_{1}\wedge \ldots \wedge g_{m})$ is the sum of all the functions obtained from $g_{1}\wedge \ldots \wedge g_{m}$ by replacing~$k$ of the $g_{i}$'s by their projections on $H^{\perp}$, and the others by their projection on~$H$.

\begin{theorem}\label{thm:projection-continuous}
For all $m\ge 1$, and all $f_1, \ldots, f_m\in E\cap \cC(S)$, we have
\begin{align}\label{eq:first-moment}
\E[\rP_{\X} f_1 \wedge \ldots \wedge \rP_{\X} f_m ]&= \proj{H} f_1\wedge \ldots \wedge \proj{H} f_m\, ,\\
\label{eq:second-moment}
{\rm Var}(\rP_{\X} f_1 \wedge \ldots \wedge \rP_{\X} f_m)&=\sum_{k=1}^{m} \sbinom{n-m+k}{k} \|{\sf \Pi}_{k}(f_{1}\wedge \ldots \wedge f_{m})\|^{2}.
\end{align}
\end{theorem}

The variance in the second assertion is that of a random element of $L^{2}(S^{m},\frac{1}{m!}\lambda^{\otimes m})$, that is, to be explicit, and in view of the first assertion,
\[{\rm Var}(\rP_{\X} f_1 \wedge \ldots \wedge \rP_{\X} f_m)= \E\Big[\big\lVert \rP_{\X} f_1 \wedge \ldots \wedge \rP_{\X} f_m -\proj{H} f_1\wedge \ldots \wedge \proj{H} f_m\big\rVert^2_{L^{2}(S^{m},\frac{1}{m!}\lambda^{\otimes m})}\Big].\]
Further note that the quadratic identity \eqref{eq:second-moment} may be polarized to obtain information on covariances.

Given the remark made after Lemma~\ref{lem:decomposition}, one can view Theorem~\ref{thm:projection-continuous} as providing a statistical estimator of part of the Euclidean structure of $E$ given $H$ and a realisation $\X$.

When $m=1$, this is the theorem of Ermakov--Zolotukhin \cite{EZ}, rephrased by \cite{Bardenet}:
\[\E[\rP_{\X}f]=\proj{H}f \ \text{ and }\ {\rm Var}(\rP_{\X}f)=n\|\proj{H^{\perp}}f\|^{2}.\]

In order to prove Theorem~\ref{thm:projection-continuous}, we will use the following generalization of Cramer's formula, which surprisingly enough, we have not encountered in our undergraduate linear algebra class.

For all integers $n$ and $m$, we denote by $\set{n}$ the set $\{1,\ldots,n\}$ and by $\Part_{m}(\set{n})$ the set of its subsets with $m$ elements. Given a $p\times q$ matrix $M$ and two subsets $I\subseteq \set{p}$ and $J\subseteq \set{q}$, we define 
\[M^{I}_{J}=(M_{ij})_{i\in I, j\in J} \ \text{ and } \ M^{I}=M^{I}_{\set{q}}.\]

\begin{proposition}[Cramer's identity for minors]\label{prop:cramer}
Let $1\le m\le n$ be two integers. 
Let $M$ be an $n\times n$ invertible square matrix, and $F$ an $n\times m$ rectangular matrix. Let $A$ be the $n\times m$ rectangular matrix solving $MA=F$. 
Then for all $I\in \Part_{m}(\set{n})$, the $m\times m$ submatrix $A^{I}$ has determinant
\begin{equation}
\det A^I = (\det M)^{-1} \det M_{[I\leftarrow F]}\, , 
\end{equation}
where $M_{[I\leftarrow F]}$ is the $n\times n$ square matrix obtained by replacing in $M$ the columns indexed by $I$ by the columns of the matrix $F$.
\end{proposition}

If $I=\{i_{1}<\ldots<i_{m}\}$, then $(M_{[I\leftarrow F]})_{ij}=M_{ij}$ for $j\notin I$, and $(M_{[I\leftarrow F]})_{ij}=F_{ik}$ for $j=i_{k}$.

\begin{proof}
Let us write $A=M^{-1}F$ and use the Cauchy--Binet formula:
\[\det A^I = \sum_{J\in \Part_{m}(\set{n})} \det (M^{-1})^I_{J} \det F^J\,.\]
Now, by Jacobi's complementary minor formula, 
\[\det (M^{-1})^I_{J} = (-1)^{\sum_{i\in I} i+\sum_{j\in J}j}\; (\det M)^{-1} \det M^{J^{c}}_{\, I^{c}}\,.\]
Combining the two previous equations and checking signs, we now recognize the Laplace expansion of $\det M_{[I\leftarrow F]}$ with respect to all columns in $I$:
\[\det A^I =(\det M)^{-1} \sum_{J\in \Part_{m}(\set{n})} (-1)^{\sum_{i\in I} i+\sum_{j\in J}j} \; \det M^{J^{c}}_{I^{c}} \det F^J =(\det M)^{-1} \det M_{[I\leftarrow F]}\,,\]
which concludes the proof.
\end{proof}

\begin{proof}[Proof of Theorem \ref{thm:projection-continuous}]
Let $(\varphi_i)_{1\le i\le n}$ be an orthonormal basis of $H$. Let $X=(x_1, \ldots, x_n)\in S^{n}$ be such that $\det (\varphi_j(x_i)_{1\leq i,j \leq n})\ne 0$. 
Let us introduce the following matrices:
\begin{itemize}
\item \begin{tabular}{p{2.5cm} c} $M=(\varphi_j(x_i))$ & $1\leq i,j \leq n$ \end{tabular},
\item \begin{tabular}{p{2.5cm} c} $F=(f_j(x_i))$ & $1\le i \le n, 1\le j\le m$ \end{tabular},
\item \begin{tabular}{p{2.5cm} c} $A=(\alpha_{ij})$ &  $1\le i\le n, 1\le j\le m$ \end{tabular}, the solution to $MA=F$,
\item \begin{tabular}{p{2.5cm} c} $G=(\langle \varphi_{i},f_{j}\rangle)$ & $1\le i\le n, 1\le j\le m$ \end{tabular}.
\end{itemize}
For each $I=\{i_{1}<\ldots<i_{k}\}\subseteq \{1,\ldots,n\}$, let us write $\varphi_{I}=\varphi_{i_{1}}\wedge \ldots \wedge \varphi_{i_{k}}$.

For each $1\le i\le m$, we have
\[\rP_X f_i = \sum_{k=1}^n A_{ki}\; \varphi_k \ \text{ and } \ \proj{H} f_i = \sum_{k=1}^n G_{ki}\; \varphi_k,\]
so that 
\begin{equation}\label{eq:ph-basis}
\rP_{X} f_1 \wedge \ldots \wedge \rP_{X} f_m=\sum_{I\in \Part_{m}(\set{n})} \det A^{I}\,  \varphi_{I}  \ \text{ and } \  \proj{H}f_1 \wedge \ldots \wedge \proj{H} f_m=\sum_{I\in \Part_{m}(\set{n})} \det G^{I} \, \varphi_{I}.
\end{equation}
In order to prove the first assertion of the theorem, namely \eqref{eq:first-moment}, we are thus left to show that for all $I\in \Part_{m}(\set{n})$, we have
\begin{equation}\label{eq:expectation}
\E\big[\det A^{I}\big]=\det G^{I},
\end{equation}
where we view $A$ as a function of the subset $X\subseteq S$ and the expectation is with respect to $\mu$.

By Proposition~\ref{prop:cramer}, we can write $\det A^I =(\det M)^{-1}\det M_{[I\leftarrow F]}$. Using the form \eqref{eq:densitydpp} of the density of $\mu$ and the Andreieff--Heine identity~\eqref{eq:AH}, we find
\begin{align*}
\E\big[\det A^{I}\big] & = \frac{1}{n!} \int_{S^n} \det A^I \; \vert\det M\vert^2 \; d\lambda^{\otimes n}\\
& = \frac{1}{n!} \int_{S^n} \det M_{[I\leftarrow F]} \; \overline{\det M} \; d\lambda^{\otimes n}\\
& = \det \big(\langle \varphi_a, \psi_{I,b}\rangle\big)_{1\le a,b\le n}\,,
\end{align*}
where $(\psi_{I,1},\ldots,\psi_{I,n})$ is the list $(\varphi_{1},\ldots,\varphi_{n})$ in which the terms labelled by elements of $I$ have been replaced by $f_{1},\ldots,f_{m}$. In symbols, $\psi_{I,b}=\varphi_b$ if $b\notin I$ and $\psi_{I,b}=f_k$ if $I=\{i_{1}<\ldots<i_{m}\}$ and $b=i_{k}$.

The last determinant is, up to conjugation by a permutation matrix, that of a $2\times 2$ block-triangular matrix. One of the diagonal blocks of this matrix is the identity, and the other is $G^{I}$.
Thus, its determinant is equal to $ \det G^{I}$, which proves \eqref{eq:expectation} and thus \eqref{eq:first-moment}.
\medskip

We now turn to the computation of the variance. An important observation is that the family $\{\varphi_{I} : I\in \Part_{m}(\set{n})\}$ is orthonormal in $L^{2}(S^{m},\tfrac{1}{m!}\lambda^{\otimes m})$. Thus, using \eqref{eq:ph-basis}, Pythagoras' theorem, and \eqref{eq:first-moment}, we find that 
\begin{equation}\label{eq:varproof}
{\rm Var}(\rP_{\X} f_1 \wedge \ldots \wedge \rP_{\X} f_m)=\sum_{I\in \Part_{m}(\set{n})} \E\big[(\det A^I)^2\big]- \big\lVert \proj{H} f_1 \wedge\ldots\wedge\proj{H} f_m \big\rVert^2.
\end{equation}

Using the same strategy as before, we compute, for each set $I$ of cardinality $m$,
\begin{align*}
\E\big[\vert\det A^I \vert^2\big] & = \frac{1}{n!} \int_{S^n} \vert \det A^I \vert^2 \; \vert\det M\vert^2 \; d\lambda^{\otimes n}\\
& = \frac{1}{n!} \int_{S^n} \vert \det M_{[I\leftarrow F]} \vert^2 \; d\lambda^{\otimes n}\\
& = \det (\langle \psi_{I,a}, \psi_{I,b}\rangle)_{1\le a,b\le n}.
\end{align*}

The last matrix has a simple block structure coresponding to the partition $\set{n}=I\sqcup I^{c}$, in which the block indexed by $(I^{c},I^{c})$ is the identity. The Schur complement formula thus gives
\begin{align*}
\det (\langle \psi_{I,a}, \psi_{I,b}\rangle)_{1\le a,b\le n}&=\det\big(\langle f_{i},f_{j}\rangle_{1\leq i,j\leq m} - \big(\langle f_{i},\varphi_{b}\rangle\big)_{1\leq i \leq m, b\in I^{c}}\big(\langle \varphi_{a},f_{j}\rangle\big)_{a\in I^{c},1\leq j \leq m}\big)\\
&=\det \left(\langle f_i, ({\rm Id}-\proj{H_{I^{c}}}) f_j \rangle\right)_{1\le i,j \le m}\\
&=\det \big(\langle f_i, (\proj{H^\perp}+\proj{H_{I}}) f_j \rangle\big)_{1\le i,j \le m}\,,
\end{align*}
where for all $J\subseteq \{1,\ldots, n\}$, we set $H_J=\Vect(\varphi_j, j\in J)$.
Using the Andreieff--Heine identity, we rewrite this determinant as
\[\E\big[\vert\det A^I \vert^2\big]=\big\langle f_1\wedge\ldots \wedge f_m, (\proj{H^\perp}+\proj{H_{I}}) f_1\wedge\ldots\wedge (\proj{H^\perp}+\proj{H_{I}})f_m\big\rangle\]
and what we need now is to sum this quantity over all $I\in \Part_{m}(\set{n})$. 

For each $i\in \{1,\ldots,m\}$, let us decompose $f_{i}$ as $f_{i,0}+f_{i,1}+\ldots+f_{i,n}$, where $f_{i,0}=\proj{H^{\perp}}f_{i}$ and for all $j\in \{1,\ldots,n\}$, $f_{i,j}=\langle \varphi_{j},f_{i}\rangle \varphi_{j}$. By multilinearity, we find
\begin{equation*}\label{eq:asommer}
\E\big[\vert\det A^I\vert^2\big]=\sum_{j_{1},\ldots,j_{m}=0}^{n} \big\langle f_1\wedge\ldots \wedge f_m, \underbrace{(\proj{H^\perp}+\proj{H_{I}}) f_{1,j_{1}}\wedge\ldots\wedge (\proj{H^\perp}+\proj{H_{I}})f_{m,j_{m}}}_{R}\big\rangle.
\end{equation*}

Let us call $R$ the function in the right-hand side of the scalar product.
If among the integers $j_{1},\ldots,j_{m}$ two are positive and equal, then $R$ vanishes, and so does the corresponding term of the sum. Let us now assume that the positive indices among $j_{1},\ldots,j_{m}$ are pairwise distinct, and let us list them as $\{l_{1},\ldots,l_{m-k}\}$, where $k=\1_{\{j_{1}=0\}}+\ldots + \1_{\{j_{m}=0\}}$. We make three observations. Firstly, for $R$ not to be zero, it is necessary that $\{l_{1},\ldots,l_{m-k}\}\subseteq I$. Secondly, if this condition is satisfied, then $R=f_{1,j_{1}}\wedge \ldots \wedge f_{m,j_{m}}$, and in particular does not depend on~$I$. Finally, the condition $\{l_{1},\ldots,l_{m-k}\}\subseteq I$ is verified for $\binom{n-m+k}{k}$ subsets $I$ of $\{1,\ldots,n\}$ with $m$ elements.
Putting these observations together, we find
\[\sum_{I\in \Part_{m}(\set{n})} \E\big[\vert\det A^I\vert^2\big]=\sum_{j_{1},\ldots,j_{m}} \sbinom{n-m+k}{k} \big\langle f_1\wedge\ldots \wedge f_m, f_{1,j_{1}}\wedge \ldots \wedge f_{m,j_{m}}\big\rangle.\]
The sum runs over those $j_{1},\ldots,j_{m}$ between $0$ and $n$ among which no two are positive and equal, but lifting this condition only adds null terms to the sum. Therefore, we let $j_{1},\ldots,j_{m}$ run freely between $0$ and $n$, and $k$ is the number of them that are zero.

Let us sort the terms of the last sum according to which of the indices $j_{1},\ldots,j_{m}$ are zero and which are not: calling $B$ the set $\{p\in \set{m} : j_{p}=0\}$ and with the notation $H^{0}=H$ and $H^{1}=H^{\perp}$, this resummation yields
\begin{equation}\label{eq:varproof2}
\sum_{I\in \Part_{m}(\set{n})} \E\big[\vert\det A^I\vert^2\big]=\sum_{k=0}^{m} \sbinom{n-m+k}{k}  \Big\langle f_1\wedge\ldots \wedge f_m, \sum_{B\in \Part_{k}(\set{m})} \proj{H^{\1_{B}(1)}} f_{1}\wedge \ldots \wedge \proj{H^{\1_{B}(m)}} f_{m}\Big\rangle.
\end{equation}
The sum over $B$ yields exactly the function ${\sf \Pi}_{k}(f_{1}\wedge \ldots \wedge f_{m})$. The result follows from the orthogonality of the decomposition \eqref{eq:decL2} and the observation that the term corresponding to $k=0$ is exactly the last term of \eqref{eq:varproof}.
\end{proof}

\medskip

\noindent{\bf Acknowledgments.} We thank the two referees for interesting and helpful comments.  

%%%%%%%%%%%%%%
%%%   Bibliographie   %%%
%%%%%%%%%%%%%%%

\def\@rst #1 #2other{#1}
\renewcommand\MR[1]{\relax\ifhmode\unskip\spacefactor3000 \space\fi \MRhref{\expandafter\@rst #1 other}{#1}}
\renewcommand{\MRhref}[2]{\href{http://www.ams.org/mathscinet-getitem?mr=#1}{MR#1}}
%%%%%%%%
\newcommand{\arXiv}[1]{\href{http://arxiv.org/abs/#1}{arXiv:#1}}
%%%%%%%%
%\newcommand{\zbMath}{\href{http://www.zbmath.org/?q=#1}{ZB#1}}
%%%%%%%%%

\bibliographystyle{hmralphaabbrv}
\bibliography{mean-projection}

\end{document}